\documentclass[]{article}
\usepackage{amssymb,amsmath,amsthm,bm,citesort,setspace,mathrsfs,nicefrac}
\usepackage{url,color}
\usepackage{pdfsync}
\input cyracc.def

\newcommand{\N}{\mathbf{N}}
\newcommand{\Z}{\mathbf{Z}}

\newcommand{\R}{\mathbf{R}}

\newcommand{\be}{\begin{equation}}
\newcommand{\ee}{\end{equation}}

\newcommand{\lip}{\text{\rm Lip}_\sigma }

\renewcommand{\P}{\mathrm{P}}
\newcommand{\E}{\mathrm{E}}

\newcommand{\1}{\boldsymbol{1}}
\renewcommand{\d}{{\rm d}}

\newcommand{\e}{{\rm e}}
\renewcommand{\geq}{\geqslant}
\renewcommand{\leq}{\leqslant}
\renewcommand{\ge}{\geqslant}
\renewcommand{\le}{\leqslant}

\author{Daniel Conus\\Lehigh University
\and Mathew Joseph\\University of Utah
\and Davar Khoshnevisan\\University of Utah
\and  Shang-Yuan Shiu\\Academica Sinica}
\title{Initial measures for the stochastic heat equation\thanks{%
	Research supported in part by 
	the NSFs grant DMS-0747758 (M.J.) and DMS-1006903 (D.K.).}}
\date{October 18, 2011}
\newtheorem{stat}{Statement}[section]
\newtheorem{proposition}[stat]{Proposition}
\newtheorem{corollary}[stat]{Corollary}
\newtheorem{theorem}[stat]{Theorem}
\newtheorem{lemma}[stat]{Lemma}
\theoremstyle{definition}

\newtheorem{example}[stat]{Example}

\numberwithin{equation}{section}
\begin{document}
\maketitle
\begin{abstract}
	We consider a family of nonlinear stochastic heat equations
	of the form $\partial_t u=\mathcal{L}u + \sigma(u)\dot{W}$,
	where $\dot{W}$ denotes space-time white noise,
	$\mathcal{L}$ the generator of a symmetric L\'evy process on $\R$,
	and $\sigma$ is Lipschitz continuous and zero at 0. We show that
	this stochastic PDE has a random-field solution for every finite initial
	measure $u_0$. Tight a priori bounds on the moments of the solution
	are also obtained.
	
	In the particular case that $\mathcal{L}f
	=cf''$ for some $c>0$, we prove that if $u_0$ is a finite measure
	of compact support, then the solution is with probability
	one a bounded function for all times $t>0$.\\
	

\end{abstract}

\section{Introduction}
Consider the stochastic heat equation
\begin{equation}\label{heat0}
	\frac{\partial}{\partial t} u_t(x) = \frac{\varkappa}{2}
	\frac{\partial^2}{\partial x^2} u_t(x) + \sigma\left(
	u_t(x)\right)\dot{W}_t(x),
\end{equation}
where $\sigma:\R\to\R$ is a Lipschitz function that satisfies
\begin{equation}
	\sigma(0)=0,
\end{equation}
 $\dot{W}$ denotes space-time white noise. The constant
$\varkappa/2$ is a viscosity constant, and is assumed to
be strictly positive throughtout. It is well known that \eqref{heat0} has a
random-field solution if, for example, $u_0$ is a bounded
and measurable function.

Now suppose that $u_0:\R\to\R_+$ is in fact bounded uniformly away from
zero, as well as infinity. We have shown recently \cite{CJK}
that $x\mapsto u_t(x)$ is a.s.\ unbounded for all $t>0$
under various conditions on $\sigma$. In particular, if
$\sigma(x)=cx$ for a constant $c>0$---this is the so called
\emph{parabolic Anderson model} \cite{CM94}---then
our results \cite{CJK} implies that with probability one,
\begin{equation}\label{chaos}
	0<\limsup_{|x|\to\infty} \frac{\log u_t(x)}{(\log |x|)^{\nicefrac23}}
	<\infty.
\end{equation}
For instance, \eqref{chaos} holds for the parabolic Anderson model
in the physically-relevant case that
$u_0(x)\equiv 1$. Another physically-relevant case is when
$u_0=\delta_0$ is point mass at 0 and we continue to consider the parabolic
Anderson model; this case arises in the study of directed random
polymers (see Kardar \cite{Kardar}). 
G. Ben Arous, I. Corwin and J. Quastel have independently asked us whether \eqref{chaos}
continues to hold in that case (private communications). One of the goals
of the present articles is to prove that the answer to this question is ``no.''
In fact, we have the following much more general fact, which
is a corollary to the development of this paper.

\begin{theorem}\label{th:NoChaos}
	If $\sigma(0)=0$ and $u_0$ is a finite measure of compact support, then
	$\sup_{x\in\R} u_t(x)=\sup_{x\in\R} |u_t(x)|<\infty$ a.s.\ for all $t>0$.
\end{theorem}

\section{Some background material}

We begin by recalling some well-known facts; also, we use this
opportunity to  set forth some notation that will be used
consistently in the sequel.

\subsection{White noise}
Throughout let $W:=\{W_t(x)\}_{t\ge 0,x\in\R}$
denote a two-parameter Brownian sheet indexed by
$\R_+\times\R$; that is, $W$ 
is a two-parameter mean-zero Gaussian process with covariance
\begin{equation}	
	\textnormal{Cov}\left( W_t(x)\,, W_s(y) \right) =
	\min(s\,,t)\min(|x|\,,|y|)\1_{(0,\infty)}(xy),
\end{equation}
for all $s,t\ge 0$ and $x,y\in\R$. The space-time mixed
derivative of $W_t(x)$ is denoted by $\dot{W}_t(x) :=
\partial^2 W_t(x)/(\partial t\,\partial x)$ and called
\emph{space-time white noise}. Space-time white noise is a generalized
Gaussian random field with mean zero and covariance measure
\begin{equation}
	\E\left[ \dot{W}_t(x)\dot{W}_s(y)\right] =\delta_0(t-s)
	\delta_0(x-y).
\end{equation}

\subsection{L\'evy processes}
Let $X:=\{X_t\}_{t\geq 0}$ denote a symmetric L\'evy process on $\R$.
That is, $t\mapsto X_t$ is [almost surely] a right-continuous
random function with left limits at every $t>0$ whose increments
are independent, identically distributed and symmetric. It is well known
that $X$ is a strong Markov process; see Jacob  \cite{Jacob}
for this and all of the analytic theory of L\'evy processes that we will
require here and throughout. We denote the
infinitesimal generator of $X$ by $\mathcal{L}$. According to the
L\'evy--Khintchine formula, the law of the process
$X$ is characterized by its \emph{characteristic exponent}; that is a function
$\Psi:\R\to\mathbf{C}$ that is determined via the identity
$\E\exp(i\xi\cdot X_t)=\exp(-t\Psi(\xi))$, valid
for all $t\geq 0$ and $\xi\in\R$. Elementary arguments show
that, because $X$ is assumed to be symmetric, the characteristic
exponent $\Psi$ is a nonnegative---in particular real valued---symmetric function. 
For reasons that will become apparent later on, we will be interested
only in L\'evy processes that satisfy the following:
\begin{equation}\label{eq:TD}
	\int_{-\infty}^\infty \e^{-t\Psi(\xi)}\,\d\xi<\infty
	\qquad\text{for all $t>0$}.
\end{equation}
In such a case, the inversion formula for Fourier transforms applies
and tells us that $X$ has transition densities $p_t(x)$ that can be defined by
\begin{equation}\label{p}
	p_t(x) := \frac{1}{2\pi} \int_{-\infty}^\infty
	\e^{-ix\cdot \xi - t\Psi(\xi)}\,\d \xi
	\qquad(t>0~,~x\in\R).
\end{equation}
Note that the function $(t\,,x)\mapsto p_t(x)$ is  continuous
uniformly on $(\eta\,,\infty)\times\R$ for every $\eta>0$. 

Let us note two important consequences of the preceding 
formula for transition densities:
\begin{enumerate}
	\item $p_t(x)\leq p_t(0)$ for all $t>0$ and $x\in\R$; and
	\item $t\mapsto p_t(0)$ is nonincreasing.
\end{enumerate}
We will appeal to the these properties without 
further mention.

Throughout we assume also that the transition densities of the
L\'evy process $X$ satisfy the following regularity condition:
\begin{equation}\label{cond:k}
	\Theta  := \sup_{t>0} \left[\frac{p_{t/2}(0)}{p_t(0)}
	\right] <\infty.
\end{equation}
Because $p_t(0)\ge p_t(x)$ and $\int_{-\infty}^\infty p_t(x)\,\d x=1$,
it follows $p_t(0)>0$ and hence $\Theta$ is well defined [though it could in principle
be infinity when $X$ is a general symmetric L\'evy process]. 

Let us mention one example very quickly before we move on.

\begin{example}
	Let $X$ denote a one-dimensional standard Brownian motion. Then,
	$X$ is a symmetric L\'evy process with transition densities
	given by $p_t(x) := (2\pi t)^{-\nicefrac 12}\exp\{ -x^2/(2t)\}$
	for $t>0$ and $x\in\R$. In this case, we may note also that
	$\mathcal{L}f=(\nicefrac12)f''$, $\Psi(\xi)=\|\xi\|^2/2$, and
	$\Theta=\sqrt 2.$\qed
\end{example}

\section{The Main result}
Our main goal is to study the nonlinear stochastic heat equation
\begin{equation}\label{heat}
	\frac{\partial}{\partial t} u_t(x) = (\mathcal{L}u_t)(x) +
	\sigma(u_t(x))\dot{W}_t(x)\qquad \text{for }
	t>0,\, x\in\R,
\end{equation}
where:
\begin{enumerate}
	\item $\mathcal{L}$ is the generator of our symmetric
		L\'evy process $\{X_t\}_{t \geq 0}$;
	\item $\sigma:\R\to\R$ is Lipschitz continuous with Lipschitz constant 
		${\rm Lip}_\sigma$; and 
	\item $\sigma(0)=0$. 
\end{enumerate}
As regards the initial data,
we will assume here and throughout that 
\begin{equation}
	\text{$u_0$ is a nonrandom, finite Borel measure on $\R$.}
\end{equation}

The best-studied special case of  the random heat equation
\eqref{heat} is when $\mathcal{L}f=\nu f''$ is
a constant multiple of the Laplacian. In that case,
Equation \eqref{heat} arises  for several reasons that include its connections to the
stochastic Burger's equation (see Gy\"ongy and Nualart \cite{GN}), the parabolic Anderson model (see Carmona and Molchanov \cite{CM94}) and the KPZ equation (see Kardar, Parisi and Zhang \cite{KPZ}).

One can think of the solution $u_t(x)$ to \eqref{heat} as
the expected density of particles, at place $x\in\R$ and time $t>0$,
for a system of interacting branching random walks in continuous time:
The particles move as independent L\'evy processes
on $\R$; and the particles move through an independent external random
environment that is space-time white noise $\dot{W}$. The mutual interactions 
of the particles occur through a possibly-nonlinear birthing mechanism $\sigma$.
The special case $\mathcal{L}f=\nu f''$ deals with the case that the mentioned
particles move as independent Brownian motions.

The most special example of \eqref{heat} 
is when $\sigma(x)\equiv 0$; that is
the linear heat [Kolmogrov] equation for $\mathcal{L}$, whose [weak] solution is
$u_t(x)=(p_t*u_0)(x)$. 
It is a simple exercise in Fourier analysis 
that when $\sigma(x)\equiv 0$, the solution to \eqref{heat}
exists, is unique, and is a bounded function
for all time $t>0$. Indeed,
\begin{equation}\label{bounded}
	(p_t * u_0)(x) = \int_{\R} p_t(y-x) u_0(\d y) 
	\leq p_t(0) \int_{\R} u_0(\d y) = p_t(0) u_0(\R) < \infty.
\end{equation}
Hence, $\sup_{x} (p_t * u_0)(x) \leq p_t(0) u_0(\R) < \infty$,
as was asserted.

Consider the case where the characteristic exponent $\Psi$ of our
L\'evy process $X$ satisfies the following condition: For some
[hence all] $\beta>0$,
\begin{equation}\label{Upsilon}
	\Upsilon(\beta) :=  \frac{1}{2\pi}\int_{-\infty}^\infty
	\frac{\d\xi}{\beta+2\Psi(\xi)} < \infty.
\end{equation}
It is well known that if, in addition, 
$u_0$ is a bounded and measurable \emph{function},
then \eqref{heat} has a solution that is a.s.\ unique among all possible
``natural'' candidates. This statement follows easily from the theory of
Dalang \cite{Dalang:99}, for instance. Moreover, Dalang's theory shows
also that \eqref{Upsilon} is necessary as well as sufficient for the existence
of a random-field solution to \eqref{heat} when $\sigma$ is a constant function.
This is why we assume \eqref{Upsilon} per force. 
The technical condition \eqref{eq:TD} in fact
follows as a consequence of \eqref{Upsilon}; 
see Foondun et al.\ \cite[Lemma 8.1]{FKN}.

Dalang's method proves also the following without any extra effort:
 
\begin{theorem}[Dalang \cite{Dalang:99}]\label{th:Dalang}
	Suppose $u_0$ is a random field, independent of the
	white noise $\dot{W}$, such that $\sup_{x\in\R}\E(|u_0(x)|^k)<\infty$
	for some $k\in[2\,,\infty)$.
	Then \eqref{heat} has a mild solution $\{u_t(x)\}_{t>0,x\in\R}$
	that solves the random integral equation
	\begin{equation}\label{eq:mild}
		u_t(x) = (p_t*u_0)(x) + \int_{(0,t)\times\R} p_{t-s}(y-x)
		\sigma(u_s(y))\, W(\d s\,\d y).
	\end{equation}
	Furthermore, $\{u_t(x)\}_{t>0,x\in\R}$ is a.s.-unique in the class of 
	all predictable random fields $\{v_t(x)\}_{t>0,x\in\R}$
	that satisfy:
	\begin{equation}
		\sup_{t\in(0,T)}\sup_{x\in\R}\E(|v_t(x)|^k)<\infty\qquad
		\text{for all $T>0$.}
	\end{equation}
	Finally, the random field $(t\,,x)\mapsto u_t(x)$
	is  continuous in probability.
\end{theorem}

We will not describe the proof, since all of the requisite
ideas are already in the paper \cite{Dalang:99}. However, we mention that
the reference to ``predictable'' assumes tacitly that the Brownian filtration
of Walsh \cite{Walsh} has been augmented with the sigma-algebra generated
by the random field $\{u_0(x)\}_{x\in\R}$.  The mentioned stochastic
integrals are also as defined in \cite{Walsh}.
We will need the following
variation of a theorem of Foondun and Khoshnevisan \cite{FK} also:

\begin{theorem}[Foondun and Khoshnevisan \cite{FK}]\label{th:FK}
	Suppose $u_0$ is a random field, independent of the
	white noise $\dot{W}$, such that $\sup_{x\in\R}\E(|u_0(x)|^k)<\infty$
	for every $k\in [2\,,\infty)$.
	Then the mild solution $\{u_t(x)\}_{t>0,x\in\R}$ to \eqref{heat} satisfies
	the following: For all $\epsilon>0$ there exists a finite and positive constant
	$C_\epsilon$ such that for all $t>0$ and $k\in[2\,,\infty)$,
	\begin{equation}
		\sup_{x\in\R}\E(|u_t(x)|^k) \leq 
		C_\epsilon^k \e^{(1+\epsilon)\gamma(k)t},
	\end{equation}
	where $\gamma(k)$ is defined by:
	\begin{equation}\label{eq:gamma(k)}
		\gamma(k) :=\inf\left\{ \beta>0:\
		\Upsilon(2\beta/k) < \frac{1}{4k\lip^2}\right\}.
	\end{equation}
\end{theorem}

Once again, we omit the proof, since it follows closely the ideas of
the paper \cite{FK} without making novel alterations.

The existence and uniqueness of the solution to \eqref{heat} under 
a measure-valued initial condition has been studied earlier
in various papers. For example, Bertini and Cancrini \cite{BC} 
obtain moment formulas for the parabolic Anderson model
[that is, $\sigma(x)=cx$] in the special case that 
$\mathcal{L}f = (\nicefrac{\varkappa}{2})f''$and $u_0=\delta_0$ 
is the Dirac point mass at zero. They also give sense to what a ``solution'' might mean.
Among many other things, a paper by Mueller \cite{Mueller:weak}
and also a recent work of Chen and Dalang \cite{LeC-D} imply
the fact that the Bertini--Cancrini solution can be arrived at by more direct
methods akin to those of Walsh \cite{Walsh}. 
Weak solutions to the fully-nonlinear equation
\eqref{heat} have been studied in Conus and Khoshnevisan \cite{Con_Kho}
as well.

We are now ready to state one of the main results of this paper. 

\begin{theorem}\label{th:exist:unique}
	If $\Theta <\infty$,
	then \eqref{heat} has a mild solution
	$u$ that satisfies the following for all real numbers $x\in\R$,
	$\epsilon, t>0$, and $k\in[2\,,\infty)$: There exists a positive and finite constant
	$C_\epsilon:=C_\epsilon(\Theta )$---depending 
	only on $\epsilon$ and  $\Theta $---such that
	\begin{equation}\label{eq:exist:unique}
		\E\left(|u_t(x)|^k\right)  \leq C_\epsilon^k
		\e^{(1+\epsilon)\gamma(k) t}\left\{ 1+ p_t(0)(p_t*u_0)(x)
		\right\}^{k/2},
	\end{equation}
	where
	\begin{equation}
		\gamma(k) :=\inf\left\{ \beta>0:\
		\Upsilon(2\beta/k) < \frac{1}{4k\lip^2}\right\}.
	\end{equation}
	Moreover, the solution is a.s.\ unique among all predictable
	random fields $v$ that solve \eqref{heat} and satisfy
	\begin{equation}
		\sup_{t>0}\sup_{x\in\R} \left[
		\frac{\E\left(|v_t(x)|^2\right)}{\e^{(1+\epsilon)
		\gamma(2) t}\left\{ 
		1\vee p_t(0)(p_t*u_0)(x) \right\} }\right] <\infty
		\qquad\text{for some $\epsilon>0$}.
	\end{equation}
\end{theorem}

From this we shall see that, in the particular case where 
$\mathcal{L}$ is a multiple of the Laplacian, the solution remains bounded 
for every finite time $t > 0$, as long as the finite initial measure $u_0$
has compact support. This verifies Theorem \ref{th:NoChaos}.

\section{An example}

Consider  \eqref{heat} where $\mathcal{L}=-\varkappa(-\Delta)^{\alpha/2}$
is the fractional Laplacian of index $\alpha\in(0\,,2]$,
where $\varkappa>0$ is a viscosity parameter. The operator $\mathcal{L}$ is
the generator of a symmetric stable-$\alpha$ L\'evy process
with $\Psi(\xi)\propto \varkappa|\xi|^\alpha$, where the
constant of proportionality does not depend on $(\varkappa\,,\xi)$.
It is possible
to check directly that $\Upsilon(1)<\infty$ if and only if $\alpha>1$.
Let us restrict attention to the case that $\alpha\in(1\,,2]$,
and recall that we consider only the case that $u_0$ is a finite measure.

A computation shows that $\Upsilon(\beta)\propto 
\varkappa^{-1/\alpha}\beta^{-(\alpha-1)/\alpha}$
uniformly for all $\beta>0$.
Moreover, $p_t(x)$ is a fundamental solution of the heat equation for
$\mathcal{L}$,
and $\Theta=2^{1/\alpha}$ since
$p_t(0)\propto (\varkappa t)^{-1/\alpha}$ uniformly for all $t>0$. 
Theorem \ref{th:exist:unique}
then tells us that \eqref{heat} has a unique mild solution which satisfies
the following for all real numbers $t>0$ and $k\in[2\,,\infty)$:
\begin{equation}
	\sup_{x\in\R}\E\left(|u_t(x)|^k\right) \leq C_1^k 
	\left( 1+ (\varkappa t)^{-1}\right)^{k/\alpha}\exp\left( C_2
	\frac{t k^{(2\alpha-1)/(\alpha-1)}}{\varkappa^{1/(\alpha-1)}}\right),
\end{equation}
where $C_1$ and $C_2$ are positive and finite constants that do not
depend on $(t\,,k\,,\varkappa)$. 
In other words, the large-$t$ behavior of
the $k$th moment of the solution is, as in \cite{FK}, the same as
it would be had $u_0$ been a bounded measurable function; 
that is,
\begin{equation}
	\limsup_{t\to\infty} \frac1t \log \sup_{x\in\R} 
	\|u_t(x)\|_k \leq \textnormal{const}
	\cdot \left(  k^\alpha / \varkappa \right)^{1/(\alpha-1)}.
\end{equation}
However, we also observe the small-$t$ estimate,
\begin{equation}\label{eq:small-t:stable}
	\limsup_{t\downarrow 0}\ t^{1/\alpha}
	\sup_{x\in\R}\|u_t(x)\|_k
	\leq  \textnormal{const}\cdot \varkappa^{-1/\alpha},
\end{equation}
which is a new property. Moreover, the preceding estimate is tight. Indeed,
it is not hard to see that 
\begin{equation}
	\|u_t(x)\|_k \ge \|u_t(x)\|_2 \geq (p_t*u_0)(x).
\end{equation}
Therefore, in the case that $u_0$ is a {\it positive definite} finite measure,
\begin{align}
	\sup_{x\in\R}\|u_t(x) \|_k &
		\geq (p_t*u_0)(0)	
		=\frac{1}{2\pi}\int_{-\infty}^\infty \e^{-\textnormal{const} 
		\cdot\varkappa t|\xi|^\alpha} \hat{u}_0(\xi)\,\d\xi\\\nonumber
	&= \frac{1}{2\pi (t\varkappa)^{1/\alpha}}
		\int_{-\infty}^\infty \e^{-\textnormal{const} 
		\cdot|\xi|^\alpha} \hat{u}_0\left( \frac{\xi}{(t\varkappa)^{1/\alpha}}
		\right) \,\d\xi.
\end{align}
The second inequality follows from applying Parseval's identity to
$(p_{t+\epsilon}*u_0)(x)$ and then letting $\epsilon\downarrow 0$
using Fatou's lemma. Another application of Fatou's lemma then shows that
\begin{equation}
	\liminf_{t\downarrow 0}\ t^{1/\alpha}
	\sup_{x\in\R}\|u_t(x)\|_k
	\geq  \textnormal{const}\cdot \varkappa^{-1/\alpha},
\end{equation}
as long as $u_0$ is a positive-definite finite  measure such that
$\varliminf_{\|z\|\to\infty}\hat{u}_0(z)>0$; that is,
the conclusion of the Riemann--Lebesgue
lemma does not apply to $u_0$. 
Thus, \eqref{eq:small-t:stable} is tight, as was claimed.
There are many examples of such measure $u_0$. For
instance, we can choose
$u_0=a\delta_0+\mu$,
where $a>0$
and $\mu$ is any given positive-definite finite Borel  measure on $\R$.

\section{Preliminaries}

\subsection{Some inequalities}

We recall from Foondun and Khoshnevsian \cite{FK} that
\begin{equation}
	\Upsilon(\beta) := 
	\frac{1}{2\pi}\int_{-\infty}^\infty
	\frac{\d\xi}{\beta+2\Psi(\xi)} = \int_0^\infty
	\e^{-\beta t} \|p_t\|_{L^2(\R)}^2\,\d t;
\end{equation}
This is merely a consequence of 
Plancherel's theorem. Because $X$ is symmetric we can describe $\Upsilon$
in terms of the resolvent of $X$. To this end define
\begin{equation}
	\overline{\Upsilon}(\beta) := \int_0^\infty \e^{-\beta t}
	p_t(0)\,\d t.
\end{equation}
This is the resolvent density, at zero, of the L\'evy process $X$.
Because of symmetry, 
\begin{equation}\label{eq:pL2}
	\| p_t \|_{L^2(\R)}^2 = (p_t*p_t)(0) = p_{2t}(0).
\end{equation}
Therefore,
\begin{equation}
	\Upsilon(\beta) =\frac12
	\overline{\Upsilon}(\beta/2).
\end{equation}
In particular, Dalang's condition (\cite[(26), thm. 2]{Dalang:99})  
$\Upsilon(1)<\infty$ is equivalent
to the condition that $\overline{\Upsilon}(\beta)<\infty$
for some, hence all, $\beta>0$.

We close this subsection with some convolution estimates.

\begin{lemma}\label{lem:pp}
	For all $t>0$,
	\begin{equation}
		p_t(0)\cdot\int_0^t p_r(0)\,\d r\leq 
		\int_0^t p_{t-s}(0)p_s(0)\,\d s \leq 2\Theta \cdot
		p_t(0)\cdot\int_0^t p_r(0)\,\d r.
	\end{equation}
\end{lemma}

As it turns out, the preceding simple-looking result is the key to our analysis of
existence and uniqueness, because it tells us that
\begin{equation}
	\int_0^t \frac{p_{t-s}(0)p_s(0)}{p_t(0)}\,\d s \xrightarrow{\hspace*{.15in}} 0
	\qquad\text{as $t\downarrow 0$},
\end{equation}
at sharp rate $\int_0^t p_r(0)\,\d r$.

\begin{proof}
	The first inequality holds simply because $p_s(0)\geq p_t(0)$
	for all $s\in(0\,,t)$. For the second one, we split
	$\int_0^t p_{t-s}(0)p_s(0)\,\d s$ into two parts: $\int_0^{t/2}$
	and $\int_{t/2}^t$. Note that $p_{t-s}(0)\leq p_{t/2}(0)\le
	\Theta p_t(0)$ when $s\in(0\,,t/2)$; and
	$p_s(0)\leq p_{t/2}(0)\le\Theta  p_t(0)$ when $s\in(t/2\,,t)$.
	The lemma follows from these observations.
\end{proof}

Let $\circledast$  denote space-time convolution; that is,
\begin{equation}\label{circledast}
	(f\circledast g)_t(x) := \int_0^t \d s\int_{-\infty}^\infty
	\d y\ f_{t-s}(x-y) g_s(y),
\end{equation}
whenever $f,g:(0\,,\infty)\times\R\to\R_+$ are both
measurable.

\begin{lemma}\label{lem:star}
	For all $t>0$, $x\in\R$, and $n\geq 1$,
	\begin{equation}\label{eq:star}
		\left( \underbrace{p^2 \circledast \cdots \circledast p^2}\limits_{\text{%
		$n$ times}} \right)_t\hskip-.06in(x) \leq 
		\left(2\Theta 
		\int_0^t p_s(0)\,\d s\right)^{n-1}\cdot p_t(0)p_t(x).
	\end{equation}
\end{lemma}

\begin{proof}
	The result holds trivially when $n=1$. Let us suppose that 
	\eqref{eq:star}
	is valid for $n=m$; we prove that \eqref{eq:star} is
	valid also for $n=m+1$. Note that
	\begin{align}\nonumber
		\mathcal{T}_{m+1} &:=
			\left(\,
			\underbrace{p^2 \circledast \cdots \circledast p^2}\limits_{\text{%
			$m+1$ times}} \right)_t(x) \\
		&= \int_0^t\d s
			\int_{-\infty}^\infty\d y\ 
			\left( \underbrace{p^2 \circledast \cdots \circledast p^2}\limits_{\text{%
			$m$ times}} \right)_{t-s}\hskip-.22in(y)\ p_s^2(x-y)\\\nonumber
		&\leq \int_0^t \d s \left( 2\Theta 
			\int_0^{t-s} p_r(0)\,\d r
			\right)^{m-1} p_{t-s}(0)\,\int_{-\infty}^\infty \d y\ p_{t-s}(y) p_s^2(x-y)\\
			\nonumber
		&\leq \left( 2\Theta  \int_0^t p_r(0)\,\d r
			\right)^{m-1}\int_0^t \d s\ p_{t-s}(0)\,
			\int_{-\infty}^\infty \d y\ p_{t-s}(y) p_s^2(x-y).\nonumber
	\end{align}
	Since $p_s^2(x-y)\leq p_s(0)p_s(x-y)$, the Chapman--Kolmogorov equation
	implies that
	\begin{equation}
		\mathcal{T}_{m+1} \le
		\left( 2\Theta  \int_0^t p_r(0)\,\d r
		\right)^{m-1}\cdot \int_0^t p_{t-s}(0)p_s(0)\,\d s \cdot 
		p_t(x),
	\end{equation}
	and the result follows from this, Lemma \ref{lem:pp}, and
	induction.
\end{proof}

\begin{lemma}\label{lem:star2}
	For all $t>0$, $x\in\R$, and $n\geq 1$,
	\begin{equation}\begin{split}
		&
			\left(\left( \underbrace{p^2 
			\circledast \cdots \circledast p^2}\limits_{\text{%
			$n$ times}} \right) \circledast (p_{\bullet}*u_0)^2\right)_t(x)
			\\
		&\hskip1.3in\leq  u_0(\R) \left(2\Theta 
			\int_0^t p_s(0)\,\d s\right)^n \cdot p_t(0) (p_t*u_0)(x).
	\end{split}\end{equation}
\end{lemma}

\begin{proof}
	For every nonnegative function $f$, \eqref{circledast} and 
	Lemma \ref{lem:star} together imply that
	\begin{equation}\begin{split}
		\lefteqn{\left(\left( \underbrace{p^2 \circledast 
			\cdots \circledast p^2}\limits_{\text{%
			$n$ times}} \right)\circledast f\right)_t (x)} 
			 \\
		& =  \int_{0}^{t} \d s \int_{\R} \d y 
			\left( \underbrace{p^2 \circledast \cdots \circledast p^2}
			\limits_{\text{$n$ times}} \right)_s(y) 
			f_{t-s}(x-y)  \\
		& \leq  \int_{0}^{t} \d s \left(2\Theta \int_{0}^{s} p_r(0)\, \d r
			\right)^{n-1} p_s(0) \int_{\R} \d y \, p_s(y) f_{t-s}(x-y)  \\
		& \leq  \left(2\Theta \int_{0}^{t} p_r(0)
			\, \d r\right)^{n-1} \int_{0}^{t} \d s \ p_s(0) (p_s*f_{t-s})(x).
	\end{split}\end{equation}
	We set $f_t(x) := p_t(0)(p_t*u_0)(x)$ and appeal the Chapman--Kolmogorov
	property [$p_s*p_{t-s}=p_t$] in order to obtain the following:
	\begin{equation}\begin{split}
		&\left(\left( \underbrace{p^2 \circledast 
			\cdots \circledast p^2}\limits_{\text{%
			$n$ times}} \right)\circledast (p_{\bullet}*u_0)^2\right)_t (x)
			 \\
		& \leq  u_0(\R) \left(\left( \underbrace{p^2 
			\circledast \cdots \circledast p^2}\limits_{\text{%
			$n$ times}} \right)\circledast p_{\bullet}(0)
			(p_{\bullet}*u_0)\right)_t (x)  \\
		& \leq u_0(\R) \left(2\Theta \int_{0}^{t} p_r(0)\,\d r\right)^{n-1} 
			\int_{0}^{t} p_s(0)p_{t-s}(0) \d s \cdot (p_t*u_0)(x). 
	\end{split}\end{equation}
	An application of Lemma \ref{lem:pp} completes the proof.
\end{proof}

\section{Finite-horizon estimates}

We first define a sequence $\{u^{(n)}\}_{n\in\N}$ of random fields by: 
$u^{(0)}_t(x):=0$ for all $t > 0$ and $x \in \R$. Then, for every
$n\geq 0$, we set
\begin{equation}\label{eq:Picard}
	u^{(n+1)}_t(x) := (p_t*u_0)(x) + \int_{(0,t)\times\R}
	p_{t-s}(y-x)\sigma(u_s^{(n)})\, W(\d s \,\d y).
\end{equation}
Thus, $u^{(n)}$ denotes simply the $n$th stage of a Picard 
iteration approximation to a reasonable candidate for a solution
to \eqref{heat}.

Proposition \ref{pr:h1} below will show that the random variables 
$\{u^{(n)}_t(x)\}_{n\in\N}$ are well-defined with values in 
$L^k(\P)$, for $x\in \R$ and all times $t$ that are ``reasonably small.''

For each $a > 0$, let
\begin{equation}\label{f}
	\mathfrak{g}(a) := \inf\left\{ t>0:\, \int_0^t p_r(0)\,\d r
	\geq a \right\},
\end{equation}
where $\inf\varnothing:=\infty$. Clearly, $\mathfrak{g}(a)>0$
when $a>0$. 

\begin{proposition}\label{pr:h1}
	For all integers $n \geq 0$, real numbers $k\in[2\,,\infty)$ and
	$x\in\R$,
	\begin{equation}
		\left\|u_t^{(n+1)}(x) \right\|_k
		\leq 4 \left[ 1\vee k^{1/2}\lip\right]\sqrt{ u_0(\R) p_t(0) (p_t*u_0)(x)},
	\end{equation}
	for $0< t\leq \mathfrak{g}((32\Theta [1\vee k \lip^2])^{-1})$.
\end{proposition}

\begin{proof}
	Define
	\begin{equation}
		\mu^{(n,k)}_t(x) := \left\|u^{(n)}_t(x) \right\|_k^2.
	\end{equation}
	Clearly, $\mu^{(0,k)}_t(x) \equiv 0$. Then, we have
	\begin{align}\nonumber
		&\hskip-.3in\sqrt{\mu^{(n+1,k)}_t(x)}\\
		&\leq ( p_t*u_0)(x)
			+ \left\| \int_{(0,t)\times\R}p_{t-s}(y-x)\sigma
			\left( u^{(n)}_s(y)\right)
			\, W(\d s\,\d y)\right\|_k\\\nonumber
		&\leq ( p_t*u_0)(x)
			+ \left(4k \cdot  \int_0^t\d s\int_{-\infty}^\infty\d y\
			p_{t-s}^2(y-x) \left\| \sigma\left( u^{(n)}_s(y) \right) \right\|_k^2
			\right)^{1/2}\\\nonumber
		&\leq ( p_t*u_0)(x)
			+ \left(4k \lip^2\cdot  \int_0^t\d s\int_{-\infty}^\infty\d y\
			p_{t-s}^2(y-x) \mu^{(n,k)}_s(y)
			\right)^{1/2}\\\nonumber
		&= (p_t*u_0)(x) + \left( 4k\lip^2\cdot
			\left( p^2\circledast \mu^{(n,k)} \right)_t(x) \right)^{1/2}.
	\end{align}
	Notice, that \eqref{bounded} implies that 
	$|(p_t*u_0)(x)|^2 \leq u_0(\R)p_t(0)(p_t*u_0)(x)$. 
	Now, define $C_k:= 8( 1 \vee k\lip^2)$, and note that
	\begin{equation}\begin{split}
		\mu^{(n+1,k)}_t(x)
			&\leq 2|(p_t*u_0)(x)|^2 +
			8 k \lip^2\cdot
			\left( p^2\circledast \mu^{(n,k)} \right)_t(x) \\
		&\leq C_k\left[|(p_t*u_0)(x)|^2 +
			\left( p^2\circledast \mu^{(n,k)} \right)_t(x) \right]\\
		&\ \,\vdots\\
		&\leq \sum_{j=0}^n C_k^{j+1} \left( \left( 
			\underbrace{p^2 \circledast
			\cdots \circledast p^2}\limits_{\text{$j$ times}} \right)
			\circledast (p_{\bullet}*u_0)^2\right)_t (x).
	\end{split}\end{equation}
	We apply Lemma \ref{lem:star2} to find that
	\begin{equation}\begin{split}
		\mu^{(n+1,k)}_t(x) &\leq C_k u_0(\R)
			p_t(0)(p_t*u_0)(x) \cdot
			\sum_{j=0}^n  \left( 2C_k\Theta 
			\cdot\int_0^t p_s(0)\,\d s\right)^j\\
		&\leq  2C_k u_0(\R) p_t(0) (p_t*u_0)(x),
	\end{split}\end{equation}
	provided that $t\leq \mathfrak{g}((4C_k\Theta )^{-1})$.
	This is another way to state the lemma.
\end{proof}

\begin{proposition}\label{pr:h2}
	If $0<t\leq \mathfrak{g}((16k\Theta \lip^2)^{-1})$, then for all integers
	$n\geq 0$ and for all $x\in\R$,
	\begin{equation}
		\left\|u^{(n+1)}_t(x)-u^{(n)}_t(x)\right\|_k^2 
		\leq  2^{-n} u_0(\R) p_t(0) (p_t*u_0)(x).
	\end{equation}
\end{proposition}

\begin{proof}
	Define for all $n\geq 0$, $t>0$, and $x\in\R$,
	\begin{equation}
		D^{(n+1,k)}_t(x) := \E\left( |u^{(n+1)}_t(x)-u^{(n)}_t(x)|^k\right)^{2/k} = \|u^{(n+1)}_t(x)-u^{(n)}_t(x)\|_k^2.
	\end{equation}
	Clearly,
	\begin{align}\nonumber
		D^{(n+1,k)}_t(x) &\leq 4k \int_0^t\d s\int_{-\infty}^\infty\d y\
			p_{t-s}^2(y-x) \E\left( \left|
			\sigma\left( u^{(n)}_s(y) \right) -
			\sigma\left( u^{(n-1)}_s(y) \right) \right|^k\right)^{2/k}\\\nonumber
		&\le4k\lip^2\cdot\int_0^t\d s\int_{-\infty}^\infty\d y\
			p_{t-s}^2(y-x)  D^{(n,k)}_s(y)\\\nonumber
		&=4k\lip^2\cdot \left( p^2 \circledast D^{(n,k)}\right)_t(x) \\
		&\leq \cdots\leq (4k\lip^2)^{n} \cdot \left(\,
			\underbrace{p^2 \circledast 
			\cdots \circledast p^2}\limits_{\text{$n$ times}}
			\circledast \, D^{(1,k)}\right)_t(x).
	\end{align}
	Note that for all $k \in [2,\infty)$,
	\begin{equation}
		D^{(1,k)}_s(y) = |(p_s*u_0)(y)|^2. 
	\end{equation}
	Therefore, in accord with Lemma \ref{lem:star2},
	\begin{equation}
		D^{(n+1,k)}_t(x)\leq 
		u_0(\R)\left(8k\lip^2 \Theta \int_0^t p_r(0)\,\d r\right)^n  p_t(0) (p_t*u_0)(x),
	\end{equation}
	and the result follows.
\end{proof}

\section{Proof of Theorem \ref{th:exist:unique}}
We prove Theorem \ref{th:exist:unique} in two parts: First we show
that there exists a solution up to time $\mathfrak{T}:=\mathfrak{g}((64\Theta [1\vee
\lip^2])^{-1})$. From there on,
it is easy to produce an all-time solution, starting from
time $t=\mathfrak{T}$.

Let $\{u^{(n)}\}_{n=0}^\infty$ be the described Picard iterates defined in \eqref{eq:Picard}.
Since
\begin{equation}
	0 < \mathfrak{T}  \leq 
	\mathfrak{g}\left( \left(32 \Theta \lip^2\right)^{-1}\right),
\end{equation}
Proposition \ref{pr:h2} implies that the sequence of random 
variables $\{u^{(n)}_t(x)\}_{n\in\N}$  converge in $L^2(\Omega)$ 
for every $0<t\le \mathfrak{T}$ and $x\in\R$.\footnote{As is customary,
$\Omega$ denotes the underlying sample space on which random
variables are defined, and
$L^2(\Omega)$ denotes the collection of all random vaiables on
$\Omega$ that have two finite moments.}

Define $U_t(x) := \lim_{n\rightarrow\infty} u^{(n)}_t(x)$, where the limit is taken in $L^2(\Omega)$. By default, 
$\{U_t(x);\, x\in\R,t\in(0,\mathfrak{T}]\}$ is a predictable random field such that
\begin{equation}
	\lim_{n\to\infty}
	\E\left( \left| u^{(n)}_t(x)  - U_t(x)\right|^k\right) = 0
	\quad\text{for all $x\in\R$},
\end{equation}
provided that $t\in(0\,,\mathfrak{T}_k]$, where $\mathfrak{T}_k := \mathfrak{g}\left( \left( 32k\Theta \left[ 1\vee \lip^2\right]
	\right)^{-1}\right)$. (Notice that $\mathfrak{T}_2 = \mathfrak{T}$.) Moreover, Proposition \ref{pr:h1} tells us that
\begin{equation}
	\left\| U_t(x) \right\|_k \leq 4k^{\nicefrac12} \left[1\vee\lip\right]
	\sqrt{u_0(\R) p_t(0)(p_t*u_0)(x)},
\end{equation}
for all $x\in\R$ and $t\in(0\,,\mathfrak{T}_k]$. Finally, these remarks
readily imply that
\begin{equation}
	U_t(x) = (p_t*u_0)(x) + \int_{(0,t)\times\R}
	p_{t-s}(y-x) \sigma\left( U_s(y)\right)\, 
	W(\d s\,\d y),
\end{equation}
for all $x\in\R$ and $t\in(0\,,\mathfrak{T}]$. In other words, 
$U$ is a mild solution to \eqref{heat} up to the nonrandom time
$\mathfrak{T}$. 

Next we define a space-time white noise $\dot{\mathcal{W}}$ by defining
its Wiener integrals as follows: For all $h\in L^2(\R_+\times\R)$,
\begin{equation}
	\int_{\R_+\times\R} h_s(y)\, \mathcal{W} (\d s\,\d y) :=
	\int_{(\mathfrak{T},\infty)\times\R} h_{s-\mathfrak{T}}(y)\,
	W(\d s\,\d y).
\end{equation}
In other words, $\dot{\mathcal{W}}$ is obtained from $\dot{W}$ by shifting
the time $\mathfrak{T}$ steps. Induction reveals that every
$\{u^{(n)}_t(x);\, x\in\R,\, t\in(0\,,\mathfrak{T}]\}$ 
is independent of the space-time white noise $\dot{\mathcal{W}}$. Therefore,
so is the short-time solution $\{U_t(x);\, x\in\R,\, t\in(0\,,\mathfrak{T}]\}$.

Next let $V:=\{V_t(x)\}_{x\in\R,t>0}$ denote the mild solution to
the stochastic heat equation
\begin{equation}
	\frac{\partial}{\partial t} V_t(x) = (\mathcal{L} V_t)(x) +
	\sigma(V_t(x))\dot{\mathcal{W}}_t(x),
\end{equation}
subject to $V_0(x) = U_{\mathfrak{T}}(x)$. Since $U_{\mathfrak{T}}$
is independent of the noise $\dot{\mathcal{W}}$, the preceding has
a unique solution, thanks to Dalang's theorem (Theorem \ref{th:Dalang}).
And since $\sup_{x\in\R}\|U_{\mathfrak{T}}(x)\|_2<\infty$
for all $\epsilon>0$, there exists $D_\epsilon\in(0\,,\infty)$ such that
for all $t>0$, and $x\in\R$,
\begin{equation}
	\E\left( |V_t(x)|^2\right) \leq D_\epsilon^2
	\e^{(1+\epsilon)\gamma(2)t},
\end{equation}
thanks to Theorem \ref{th:FK}.
Finally, we define for all $x\in\R$,
\begin{equation}
	u_t(x) := \begin{cases}
		U_t(x)&\text{if $t\in(0\,,\mathfrak{T}]$},\\
		V_{t-\mathfrak{T}}(x)&\text{if $t>\mathfrak{T}$}.
	\end{cases}
\end{equation}
Then it is easy to see that the random field $u$ is predictable, and is a mild
solution to \eqref{heat} for all $t>0$,  subject to initial measure being
$u_0$. Uniqueness is  a standard consequence of \eqref{eq:exist:unique}.

Let us now consider $k > 2$ and follow the same argument as above, 
but use $\mathfrak{T}_k$ instead of $\mathfrak{T}$; i.e., 
we run the solution $U$ up to time $\mathfrak{T}_k$ only and then 
keep going with the classical technique of Dalang. Then, a similar 
argument leads us to a moment estimate of order $k$ for $u_t(x)$,
thanks to another application of Theorem \ref{th:FK}. 
The solution obtained from $\mathfrak{T}_k$ is the same as the one obtained when stopping at $\mathfrak{T}$ by the uniqueness result.
\qed\\

We pause to state an immediate corollary of the proof of
Theorem \ref{th:exist:unique}, as it might be of some independent interest.
In words, the following shows that if $\sigma$ has truly-linear growth and
$\Theta <\infty$, then the solution to \eqref{heat}
has nontrivial moment Liapounov exponents.

\begin{corollary}
	Suppose $\Theta <\infty$, $\textnormal{L}_\sigma:=
	\inf_{x\in\R}|\sigma(x)/x|>0$, and $u_0$ is a finite Borel measure
	on $\R$. Then,
	\begin{equation}
		0<\limsup_{t\to\infty} \frac1t \log\E\left(|u_t(x)|^k\right)<\infty,
	\end{equation}
	for all $k\in[2\,,\infty)$.
\end{corollary}

\begin{proof}
	Let $\{V_t(x)\}_{t>0,x\in\R}$ denote the post-$\mathfrak{T}_k$ process
	used in the proof of Theorem \ref{th:exist:unique}. It suffices to prove that
	\begin{equation}
		0<\limsup_{t\to\infty} \frac1t \log\E\left(|V_t(x)|^k\right)<\infty,
	\end{equation}
	for all $k\in[2\,,\infty)$. This follows from \cite{FK}.
\end{proof}

The  proof of Theorem \ref{th:exist:unique}
is based on the idea that one can solve \eqref{heat}
up to time $\mathfrak{T}$, using the method of the present paper; and then
from time $\mathfrak{T}$ on we paste the more usual solution, shifted by time
$\mathfrak{T}$ time steps, in order to obtain a global solution to \eqref{heat}.
But in fact since the pre-$\mathfrak{T}$ and the post-$\mathfrak{T}$ solutions
are unique [a.s.], we could replace $\mathfrak{T}$ by any other time $\eta$ (not necessarily one of the $\mathfrak{T}_k$) before it
as well. The following merely enunciates these observations in the form of a 
proposition. The proof follows from the fact that the sequence $\mathfrak{T}_k$ goes to $0$ as $k$ increases. We omit the details. However, we state
this simple result explicitly, as it will be central to 
our proof of Theorem \ref{th:NoChaos}.

\begin{proposition}\label{pr:modify:eta}
	Choose and fix some $\eta\in(0\,,\mathfrak{T})$,
	and let us define the predictable random field
	$\{\bar{V}_t(x)\}_{t>0,x\in\R}$ exactly as we defined
	$\{V_t(x)\}_{t>0,x\in\R}$,
	except with $\mathfrak{T}$ replaced everywhere by $\eta$.
	Finally define $\bar{u}_t(x)$ as we did $u_t(x)$, except we replace
	$(U,V,\mathfrak{T})$ by $(U,\bar{V},\eta)$; that is,
	\begin{equation}
		\bar{u}_t(x) := \begin{cases}
			U_t(x)&\text{if $t\in(0\,,\eta]$,}\\
			\bar{V}_{t-\eta}(x)&\text{if $t>\eta$}.
		\end{cases}
	\end{equation}
	Then, the random field $\bar{u}$ is a modification of 
	the random field $u$.
\end{proposition}

\section{Stability and positivity}

Let $u$ denote the solution to \eqref{heat},  as
defined in Theorem \ref{th:exist:unique},
starting from a finite Borel measure $u_0$. We have seen that
$(p_t*u_0)(x)$ is finite for all $t>0$ and $x\in\R$ fixed. Also, for $\epsilon > 0$, let
$U^{(\epsilon)}$ denote the solution to \eqref{heat}, starting
from the [bounded and measurable] initial function $p_\epsilon*u_0$.

\begin{proposition}[Stability]\label{pr:stability}
	For every $t,\epsilon>0$,
	$u_t, U^{(\epsilon)}_t\in L^2(\Omega\times\R)$.
	Moreover, the following bound is valid
	for all $\beta$ such that $\Upsilon(\beta) \geq (2\lip^2)^{-1}$:
	\begin{equation}\begin{split}
		&\hskip-.2in\int_0^\infty\e^{-\beta t}\,\d t\int_{-\infty}^\infty\d x\
			\E\left(\left|u_t(x) - U^{(\epsilon)}_t(x) \right|^2\right)\\
		&\hskip2in\leq \frac{[u_0(\R)]^2}{\pi}\int_{-\infty}^\infty
			\frac{\left(1-\e^{-\epsilon\Psi(\xi)}\right)^2}{\beta+2\Psi(\xi)}
			\,\d\xi.
	\end{split}\end{equation}
	In particular, the left-hand side tends to zero as $\epsilon\downarrow 0$.
\end{proposition}

\begin{proof}
	Let $u^{(n)}_t(x)$ be the $n$th Picard iterate, defined in \eqref{eq:Picard}. Then,
	\begin{equation}\begin{split}	
		&\hskip-.35in\left\| u^{(n+1)}_t(x) \right\|^2_2 \\
		&\leq \left| (p_t*u_0)(x) \right|^2 + \lip^2\cdot
			\int_0^t\d s\int_{-\infty}^\infty\d y\
			p_{t-s}^2(y-x)\left\| u^{(n)}_s(y) \right\|^2_2.
	\end{split}\end{equation}
	We integrate $[\d x]$ to find that
	\begin{equation}\begin{split}	
		&\hskip-.2in\E\left(\left\| u^{(n+1)}_t \right\|^2_{L^2(\R)}\right) \\
		&\leq \left\| p_t*u_0 \right\|^2_{L^2(\R)} + \lip^2\cdot
			\int_0^t \|p_{t-s}\|_{L^2(\R)}^2\cdot
			\E\left(\left\| u^{(n)}_s \right\|^2_{L^2(\R)}\right)\,\d s\\
		&= \left\| p_t*u_0 \right\|^2_{L^2(\R)} + \lip^2\cdot
			\int_0^t p_{2(t-s)}(0)\cdot
			\E\left(\left\| u^{(n)}_s \right\|^2_{L^2(\R)}\right)\,\d s;
	\end{split}\end{equation}
	see \eqref{eq:pL2}.
	Note that $|\hat{u}_0(\xi)|\leq u_0(\R)$, whence
	\begin{equation}
		\left\| p_t*u_0\right\|_{L^2(\R)}^2 =
		\frac{1}{2\pi}\int_{-\infty}^\infty \e^{-2t\Psi(\xi)}
		\left| \hat{u}_0(\xi)\right|^2\,\d\xi
		\leq |u_0(\R)|^2\cdot p_{2t}(0),
	\end{equation}
	thanks to Plancherel's theorem. [One can construct an alternative proof of
	this inequality, using the semigroup property of $p_t$ and the Young inequality.]
	Therefore,
	\begin{equation}\begin{split}	
		&\E\left(\left\| u^{(n+1)}_t \right\|^2_{L^2(\R)}\right) \\
		&\leq |u_0(\R)|^2\cdot p_{2t}(0) + \lip^2\cdot
			\int_0^t  p_{2(t-s)}(0)\cdot
			\E\left(\left\| u^{(n)}_s \right\|^2_{L^2(\R)}\right)\,\d s.
	\end{split}\end{equation}
	Define, for all predictable random fields $f$, the quantity
	\begin{equation}
		\mathcal{K}_t(f) := \sup_{s\in(0,t)} \left[ \frac{\E\left(
		\|f_s\|_{L^2(\R)}^2\right)}{1+  p_{2s}(0)} \right],
	\end{equation}
	in order to find that
	\begin{align}
		&\E\left(\left\| u^{(n+1)}_t \right\|^2_{L^2(\R)}\right)\\\nonumber
		&\leq |u_0(\R)|^2\cdot p_{2t}(0) + \lip^2\cdot
			\mathcal{K}_t\left(u^{(n)}\right)\cdot
			\int_0^t p_{2(t-s)}(0)\left[ 1+p_{2s}(0)\right]\,\d s\\\nonumber
		&\le|u_0(\R)|^2\cdot p_{2t}(0) + \frac{1}{2}\lip^2\cdot
			\mathcal{K}_t\left(u^{(n)}\right)\cdot
			\left(\int_0^{2t} p_s(0)\,\d s + \int_0^{2t} p_{2t-s}(0)
			p_s(0)\,\d s\right)\\\nonumber
		&\le|u_0(\R)|^2\cdot p_{2t}(0) + \frac{1}{2}\lip^2\cdot
			\mathcal{K}_t\left(u^{(n)}\right)\cdot
			\int_0^{2t} p_s(0)\,\d s\cdot \left( 1+ 2\Theta p_{2t}(0)\right);
	\end{align}
	see Lemma \ref{lem:pp}. Since $\Theta \geq 1$, this leads us
	to the following:
	\begin{equation}
		\mathcal{K}_t\left( u^{(n+1)}\right)
		\leq |u_0(\R)|^2 + \lip^2\Theta \cdot
		\mathcal{K}_t\left(u^{(n)}\right)\cdot
		\int_0^{2t} p_s(0)\,\d s.
	\end{equation}
	Recall $\mathfrak{T}$ from \eqref{T}.
	If $t\in(0\,,\mathfrak{T}/2]$, then $\int_0^{2t}p_s(0)\,\d s$ is
	certainly bounded above by $(4\Theta \lip^2)^{-1}$, whence we have
	\begin{equation}
		\mathcal{K}_t\left( u^{(n+1)}\right)
		\leq  |u_0(\R)|^2 + \frac12
		\mathcal{K}_t\left(u^{(n)}\right)\le\cdots\le
		2|u_0(\R)|^2,
	\end{equation}
	since $\mathcal{K}_t(u^{(0)})=0$. Therefore,
	\begin{equation}
		\E\left( \| u_t \|_{L^2(\R)}^2\right) \leq 2|u_0(\R)|^2\left(
		1+p_{2t}(0)\right)\qquad
		\text{for all $t\in(0\,,\mathfrak{T}/2]$}.
	\end{equation}
	One proves, similarly, that uniformly for all $\epsilon>0$,
	\begin{equation}
		\E\left( \| U^{(\epsilon)}_t \|_{L^2(\R)}^2\right) \leq 2|u_0(\R)|^2\left(
		1+p_{2t}(0)\right)\qquad
		\text{for all $t\in(0\,,\mathfrak{T}/2]$}.
	\end{equation}
	By Proposition \ref{pr:modify:eta}, the process 
	$\{u_{t+\mathfrak{T}/2}\}_{t\geq 0}$ starts from $u_{\mathfrak{T}/2}
	\in L^2(\Omega\times\R)$ and solves the shifted form of \eqref{heat}, 
	and hence is in $L^2(\R)$ for all time
	$t\ge\mathfrak{T}/2$ by Foondun and Khoshnevisan
	\cite[Theorem 1.1]{FK:AIHP}; for earlier developments along similar
	lines see Dalang and Mueller \cite{DalangMueller}. Similar remarks also apply
	to $\{U^{(\epsilon)}_t(x)\}_{x\in\R,t>0}$.
	
	Define, 
	\begin{equation}
		\mathcal{D}^{(\epsilon)}_t(x):=\E\left(\left|
		u_t(x) - U^{(\epsilon)}_t(x) \right|^2\right).
	\end{equation}
	Since $p_t*(p_\epsilon*u_0)=p_{t+\epsilon}*u_0$,
	\begin{equation}
		\mathcal{D}^{(\epsilon)}_t(x)\
		\leq \left| (p_t*u_0)(x) - (p_{t+\epsilon}*u_0)(x)\right|^2+
		\lip^2\cdot\left( p^2\circledast 
		\mathcal{D}^{(\epsilon)}\right)_t(x).
	\end{equation}
	We integrate $[\d x]$ and apply the Plancherel theorem to find that
	\begin{equation}\begin{split}
		\left\| \mathcal{D}^{(\epsilon)}_t\right\|_{L^1(\R)}
			&\le\frac{[u_0(\R)]^2}{2\pi}\int_{-\infty}^\infty \e^{-2t\Psi(\xi)}
			\left( 1 - \e^{-\epsilon\Psi(\xi)}\right)^2\,\d\xi\\
		&\hskip1in+ \lip^2\cdot\int_0^t\|p_{t-s}\|_{L^2(\R)}^2
			\cdot\left\| \mathcal{D}^{(\epsilon)}_s\right\|_{L^1(\R)}\,\d s.
	\end{split}\end{equation}
	We integrate one more time $[\exp(-\beta t)\,\d t]$ in
	order to see that
	\begin{equation}
		\mathcal{E}_\beta^{(\epsilon)}:=
		\int_0^\infty\e^{-\beta t}\left\| \mathcal{D}_t^{(\epsilon)}\right\|_{
		L^1(\R)}\,\d t
	\end{equation}
	satisfies
	\begin{equation}
		\mathcal{E}_\beta^{(\epsilon)}
		\leq \frac{[u_0(\R)]^2}{2\pi}\int_{-\infty}^\infty
		\frac{\left(1-\e^{-\epsilon\Psi(\xi)}\right)^2}{\beta+2\Psi(\xi)}\,\d\xi
		+\lip^2\Upsilon(\beta)\cdot\mathcal{E}_\beta^{(\epsilon)}.
	\end{equation}
	Pick $\beta$ large enough that $\Upsilon(\beta)\leq (2\lip^2)^{-1}$ to
	obtain the claimed inequality of the proposition. And since
	$\Upsilon(\beta)<\infty$, the final assertion about convergence to 0
	follows from this inequality and the dominated convergence theorem.
\end{proof}

\begin{proposition}[positivity]\label{pr:positivity}
	If $\sigma(0)=0$ and $u_0(\R)>0$, then $u_t(x)\geq 0$
	a.s.\ for all $t>0$ and $x\in\R$.
\end{proposition}

\begin{proof}
	Since $u_0$ is a finite measure, it follows that 
	$U^{(\epsilon)}_0(x) = (p_\epsilon*u_0)(x)\le
	p_\epsilon(0) u_0(\R)<\infty$, uniformly in $x\in\R$. According to
	Mueller's comparison principle, because
	$U^{(\epsilon)}_0(x)\geq 0$, it follows that
	$U^{(\epsilon)}_t(x)\geq 0$ a.s.\ for all $t>0$ and $x\in\R$.
	[Mueller's comparison principle \cite{Mueller} was proved originally
	in the case that $\mathcal{L}$ is proportional to the
	Laplacian. This comparison principle can be shown to hold in the
	more general setting of the present paper as well, though we admit this
	undertaking requires some effort when $\mathcal{L}$ is not
	proportional to the Laplacian.]
	
	In accord with Proposition \ref{pr:stability},
	\begin{equation}
		\P\left\{ u_t(x)\geq 0\text{ for almost every $t>0$ and $x\in\R$}\right\}=1.
	\end{equation}
	In particular, for all $\eta>0$,
	\begin{equation}
		\P\left\{ u_t(x)\geq 0\text{ for almost every $t\ge\eta$ 
		and $x\in\R$}\right\}=1.
	\end{equation}
	This shows that 
	\begin{equation}
		\P\left\{ \bar{V}_t(x)\geq 0
		\text{ for almost every $t>0$ and $x\in\R$}\right\}=1,
	\end{equation}
	where $\bar{V}$ was defined in Proposition \ref{pr:modify:eta}.
	According to Dalang's theory (Theorem \ref{th:Dalang}),
	$(t\,,x)\mapsto\bar{V}_t(x)$ is continuous in probability. Therefore,
	it follows that $\bar{V}_t(x)\geq 0$ a.s.\ for every $t>0$ and $x\in\R$;
	we note the order of the quantifiers.
	Therefore, a second application of Proposition \ref{pr:modify:eta} implies
	the proposition.
\end{proof}

\section{Proof of Theorem \ref{th:NoChaos}}
Throughout this section, we assume that 
\begin{equation}
\sigma(0)=0.
\end{equation}
We simplify the notation somewhat by assuming, without a great loss
in generality, that
\begin{equation}
	\varkappa=1.
\end{equation}
In this way, $\mathcal{L}f=(\nicefrac12)f''$ is the generator of standard
Brownian motion, and $\{u_t(x)\}_{t>0,x\in\R}$ satisfies \eqref{eq:mild}
with
\begin{equation}
	p_t(x) := \frac{\e^{-x^2/(2t)}}{(2\pi t)^{1/2}}
	\qquad\text{for $x\in\R$ and $t>0$}.
\end{equation}
The proof of Theorem \ref{th:NoChaos} uses the theory of the 
present paper, but also borrows heavily from the method of 
Foondun and Khoshnevisan \cite{FK:AIHP}. 

\begin{lemma}\label{lem:FK:lem:3.3}
	Suppose $u_0$ is a finite measure that is supported in $[-K\,,K]$
	for some $K>0$. Then for all $t>0$, $k\in[1\,,\infty)$,
	and $x\in\R$,
	\begin{equation}
		\limsup_{|x|\to\infty} \frac{1}{x^2}
		\log \E\left(|u_t(x)|^k\right) <0.
	\end{equation}
\end{lemma}

\begin{proof}
	This is essentially the same result as \cite[Lemma 3.3]{FK:AIHP}. We mention
	how to make the requisite changes to the proof of the said result in order to
	derive the present form of the lemma.
	
	Since $u_t(x)\geq 0$ a.s.\ (Proposition \ref{pr:positivity}), we obtain from
	\eqref{eq:mild} the following:
	\begin{equation}\label{eq:L1est}
		\E\left(|u_t(x)|\right) = (p_t*u_0)(x) = \int_{-K}^K
		\frac{\e^{-(x-y)^2/(2t)}}{(2\pi t)^{1/2}}\, u_0(\d y)
		\leq \text{const}\cdot \e^{-x^2/(4t)},
	\end{equation}
	using the elementary inequality: $(x-y)^2\ge(x^2/2)-K^2$, valid
	when $|y|\leq K$. 
	And because the preceding constant does not depend on $x$, we have
	for all $k\in[2\,,\infty)$ and $c\in(0\,,\infty)$,
	\begin{align}
		\E\left(|u_t(x)|^k\right) &\leq c^k +
			\left\{ \E\left( |u_t(x)|^{2k}\right)\right\}^{1/2}
			\cdot\left[ \P\{u_t(x)>c\} \right]^{1/2}\\\nonumber
		&\leq  c^k + \textnormal{const}
			\cdot\left[ \P\{u_t(x)>c\} \right]^{1/2};
	\end{align}
	this follows readily from the estimate of Theorem \ref{th:exist:unique}.
	We emphasize that the ``const'' does not depend on $(c\,,x)$. 
	Owing to \eqref{eq:L1est}, this leads us to
	\begin{equation}
		\E\left( |u_t(x)|^k\right) \leq \left[ c^k 
		+ \frac{\alpha}{\sqrt c}\e^{-x^2/(8t)}\right],
	\end{equation}
	where $\alpha\in(0\,,\infty)$ does not depend on $c$. Therefore
	we may optimize over $c>0$ in order to obtain the inequality
	\begin{equation}
		\limsup_{|x|\to\infty} \frac{1}{x^2}\log\E\left(|u_t(x)|^k\right)
		\le -\frac{k}{4(2k+1)t}.
	\end{equation}
	The lemma follows readily from this.
\end{proof}

\begin{lemma}\label{lem:modulus1}
	For all $t>0$ and $k\in[1\,,\infty)$,
	\begin{equation}
		\sup_{j\in\Z}\ \sup_{j\le x<x'\le j+1}
		\E\left( \frac{|u_t(x)-u_t(x')|^{2k}}{|x-x'|^k}\right) <\infty.
	\end{equation}
\end{lemma}

\begin{proof}
	It is not so easy to prove this result directly from \eqref{eq:mild},
	since the map $s\mapsto u_s(y)$ is  singular near
	$s=0$. Because $t>0$ is fixed in the statement of our lemma,
	we may instead apply Proposition \ref{pr:modify:eta} 
	in order to see that our lemma
	follows from the following.\\
	
	\textbf{Claim.}
	{\it Suppose $\bar{V}$ solves \eqref{heat},
	where $\bar{V}_0$ is a random field, independent of the noise, and
	$m_{\nu} := \sup_{x\in\R}
	\|\bar{V}_0(x)\|_\nu<\infty$ for all $\nu\in[2\,,\infty)$.
	Then for every fixed $t>0$ and $k\in[1\,,\infty)$, 
	there exists a positive and finite constant $K$ such that
	\begin{equation}\label{claim:V}
		\sup_{\substack{x,x'\in\R:\\|x-x'|\le 1}}
		\frac{\left\| \bar{V}_t(x) - 
		\bar{V}_t(x') \right\|_{2k}}{|x-x'|^{\nicefrac12}} \le K.
	\end{equation}
	}

	We prove the Claim [i.e., \eqref{claim:V}]
	by adapting the 
	method of proof of \cite[Lemma 3.4]{FK:AIHP} to the present setting. 
	
	First, we assert that for all fixed $t>0$ and $k\in[1\,,\infty)$,
	\begin{equation}\label{claim1}
		\left\| (p_t*\bar{V}_0)(x) - (p_t*\bar{V}_0)(x') 
		\right\|_{2k} \le \text{const}\cdot |x-x'|,
	\end{equation}
	where the constant is independent of $x, \, x'$.
	Indeed, according to the Minkowski inequality,
	\begin{equation}\begin{split}
		&\left\| (p_t*\bar{V}_0)(x) - (p_t*\bar{V}_0)(x') 
			\right\|_{2k}\\
		&\hskip1in\le \int_{-\infty}^\infty \| \bar{V}_0(x)\|_{2k}
			\left| p_t(y-x)- p_t(y-x') \right|\,\d y\\
		&\hskip1in\le m_{2k} \cdot \int_{-\infty}^\infty 
			\left| p_t(y-|x-x'|) - p_t(y) \right|\,\d y.
	\end{split}\end{equation}
	We estimate the last integral by applying the fundamental 
	theorem of calculus---using the
	fact that $p_t'(z) = -(z/t) p_t(z)$---in order to deduce \eqref{claim1}.
	
	Next we observe that, as a consequence of \eqref{eq:mild}, \eqref{claim1},
	 and the BDG inequality [using the Carlen--Kree bound \cite{CK}
	on Davis's optimal constant \cite{Davis} in the Burkholder--Davis--Gundy 
	inequality \cite{Burkholder,BDG,BG}],
	\begin{equation}\begin{split}
		&\left\| \bar{V}_t(x) - \bar{V}_t(x') \right\|_{2k} \\
		&\le \textnormal{const} \cdot |x-x'| \\
		&+\left( 8k\int_0^t\d s\int_{-\infty}^\infty\d y\
			\|\sigma(\bar{V}_s(y))\|_{2k}^2 \cdot
			\left| p_{t-s}(y-x)-p_{t-s}(y-x') \right|^2\right)^{1/2};
	\end{split}\end{equation}
	 see \cite{FK}
	for details for deriving this sort of inequality. Since $|\sigma(z)/z|\le
	\lip$, we are led to the bound
	\begin{align}\nonumber
		&\left\| \bar{V}_t(x) - \bar{V}_t(x') \right\|_{2k} \\
		&\le \textnormal{const} \cdot |x-x'| \\\nonumber
		&\ +\left( 8k\lip^2\cdot\int_0^t\d s\int_{-\infty}^\infty\d y\
			\| \bar{V}_s(y)\|_{2k}^2 \cdot
			\left| p_{t-s}(y-x)-p_{t-s}(y-x') \right|^2\right)^{1/2}.
	\end{align}
	Theorem \ref{th:FK}, applied to the present choice of
	$\mathcal{L}$, tells us that $\|\bar{V}_s(y)\|_{2k}\le\exp(csk^2)$
	for a  constant $c\in(1\,,\infty)$ that does not depend on any of
	the parameters except $\lip$. Therefore, there exists a
	finite constant $C:=C(\lip)>1$, such that
	\begin{align}
		&\left\| \bar{V}_t(x) - \bar{V}_t(x') \right\|_{2k} \\\nonumber
		&\le \textnormal{const} \cdot |x-x'| +C k^{1/2}\e^{ck^2t}
			\left( \int_0^t\d s\int_{-\infty}^\infty\d y\
			\left| p_s(y-x)-p_s(y-x') \right|^2\right)^{1/2}.
	\end{align}
	By Plancherel's formula,
	\begin{equation}\begin{split}
		\int_{-\infty}^\infty \left| p_s(y-x)-p_s(y-x') \right|^2\, \d y
			&= \frac{1}{2\pi}\int_{-\infty}^\infty \e^{-\nicefrac{\xi^2s}{2}}
			\left| \e^{-i\xi x} - \e^{-i\xi x'}\right|^2
			\d\xi\\
		&\le \frac{2}{\pi}\int_0^\infty \e^{-\nicefrac{\xi^2s}{2}}
			\left[ 1\wedge \xi |x-x'|\right]^2\,\d\xi.
	\end{split}\end{equation}
	Consequently,
	\begin{equation}\begin{split}
		&\int_0^t\d s\int_{-\infty}^\infty \d y\
			\left| p_s(y-x)-p_s(y-x') \right|^2 \\
		&\hskip1in\le \e^t\int_0^\infty \e^{-s}\,\d s
			\int_{-\infty}^\infty \d y\
			\left| p_s(y-x)-p_s(y-x') \right|^2\\
		&\hskip1in\le \frac{2\e^t}{\pi}\int_0^\infty \frac{%
			\left[ 1\wedge \xi |x-x'|\right]^2}{1+\nicefrac{\xi^2}{2}}\,\d\xi.
	\end{split}\end{equation}
	Let $\mathcal{I}$ denote the latter integral. For simplicity, let us denote $\delta=|x-x'|$. It suffices to prove that
	$\mathcal{I}\le  3\delta$; this inequality implies \eqref{claim:V} ,
	whence the lemma. In order to estimate $\mathcal{I}$ we write it
	as $\mathcal{I}_1+\mathcal{I}_2+\mathcal{I}_3$, where 
	$\mathcal{I}_1:=\int_0^1(\,\cdots)\,\d\xi$,
	$\mathcal{I}_2:=\int_1^{1/\delta}(\,\cdots)\,\d\xi$, and
	$\mathcal{I}_3:=\int_{1/\delta}^\infty(\,\cdots)\,\d\xi$.
	Note that:
	$\mathcal{I}_1 \le\delta^2 \int_0^1\xi^2\,\d\xi =\delta^2/3$;
	$\mathcal{I}_2\le \delta^2 \int_1^{1/\delta}\d\xi 
	=\delta-\delta^2$;
	and $\mathcal{I}_3\le 2\int_{1/\delta}^\infty\xi^{-2}\,\d\xi=2\delta.$
	Therefore, $\mathcal{I}\le 3\delta -(\nicefrac23)\delta^2<
	3\delta$, as asserted.
\end{proof}

Our next result follows immediately from Lemma \ref{lem:modulus1}
and a quantitative form of the Kolmogorov continuity theorem. The proof
is exactly the same as that of Ref.\ \cite[Lemma 3.6]{FK:AIHP}, and is therefore
omitted.

\begin{lemma}\label{lem:modulus2}
	For all $t>0$, $k\in[1\,,\infty)$, and $\epsilon\in(0\,,1)$, 
	\begin{equation}
		\sup_{j\in\Z}\left\| \sup_{j\le x<x'\le j+1}
		\frac{|u_t(x)-u_t(x')|^2}{|x-x'|^{1-\epsilon}}\right\|_{2k}
		<\infty.
	\end{equation}
\end{lemma}

We are ready to prove Theorem \ref{th:NoChaos}.

\begin{proof}[Proof of Theorem \ref{th:NoChaos}]
	We follow carefully the arguments of Foondun and Khoshnevisan
	\cite[(3.43) and on]{FK:AIHP}.
	
	For all $j\ge 1$, we may write \cite[(3.43)]{FK:AIHP}
	\begin{equation}
		\sup_{\log j\le x\le \log(j+1)} |u_t(x)|^6 \le
		32\left(|u_t(\log j)|^6 + (\log(j+1)-\log j)^3\,
		\Omega_j^3\right),
	\end{equation}
	where
	\begin{equation}
		\Omega_j := \sup_{\log j\le x<x'\le \log(j+1)}
		\frac{|u_t(x)-u_t(x')|^2}{|x-x'|^{\nicefrac12}}.
	\end{equation}
	Consequently,
	\begin{equation}\begin{split}
		&\E\left(\sup_{\log j\le x\le \log(j+1) }|u_t(x)|^6\right)\\
		&\hskip.8in\le 32\left( \E\left(|u_t(\log j)|^6\right) + 
			\left( \log\left[1+\frac 1j\right]\right)^3\,\E(\Omega_j^3)\right).
	\end{split}\end{equation}
	According to Lemma \ref{lem:FK:lem:3.3},
	$\E(|u_t(\log j)|^6)\le\exp(- D(\log j)^2)$ for a positive
	and finite constant $D$
	that does not depend on $j$. Lemma \ref{lem:modulus2}
	tells us that $\E(\Omega_j^3)$ is bounded uniformly in $j$.  Since 
	$\log(1+j^{-1})\le j^{-1}$, we therefore have
	\begin{equation}
		\E\left(\sup_{\log j\le x\le \log(j+1)}|u_t(x)|^6\right)
		\le \textnormal{const}\cdot
		\left( \e^{-D(\log j)^2} + j^{-3}\right),
	\end{equation}
	where ``const'' can be chosen independently of $j$. 
	Because the preceding is summable [in $j$], 
	it follows that $\sup_{x\ge 0}|u_t(x)|\in
	L^6(\Omega)$,
	whence $\sup_{x\ge 0}|u_t(x)|<\infty$
	a.s. Similarly, $\sup_{x\le 0}|u_t(x)|<\infty$ a.s.
	This completes the proof, since we know that $u_t(x)\ge 0$ a.s.\
	for all $t>0$ and $x\in\R$ (Proposition \ref{pr:positivity}),
	and $x\mapsto u_t(x)$ is continuous (Lemma \ref{lem:modulus2}).
\end{proof}

\subsection*{Acknowledgements} 
This paper owes its existence to a question independently asked by Dr. Ivan Corwin and Professors G\'erard Ben Arous and Jeremy Quastel. [Theorem \ref{th:NoChaos}
contains an answer to that question.] We thank
them for telling us about this interesting topic.

\begin{small}
\bigskip

\noindent\textbf{Daniel Conus}, Lehigh University,
	Department of Mathematics, Christmas--Saucon Hall,
	14 East Packer Avenue, Bethlehem, PA 18015
	(\texttt{daniel.conus@lehigh.edu}) \\

\noindent\textbf{Mathew Joseph} 
	and  \textbf{Davar Khoshnevisan},
	University of Utah, Department of Mathematics, 
	155 South 1400 East JWB 233, 
	Salt Lake City, UT 84112-0090\par
	\noindent(\texttt{joseph@math.utah.edu} and
	\texttt{davar@math.utah.edu})\\
	
\noindent\textbf{Shang-Yuan Shiu}
		 Institute of Mathematics, Academia Sinica, 
		6F, Astronomy-Mathematics Building,
		No.\ 1, Sec. 4, Roosevelt Road, Taipei 10617, TAIWAN
		(\texttt{shiu@math.sinica.edu.tw})\\
\end{small}

\end{document}